\documentclass[reqno,10pt, centertags,draft]{amsart}
\usepackage{amsmath,amsthm,amscd,amssymb,latexsym,upref}
\usepackage{latexsym}
\usepackage{lscape}
\usepackage{url}

\newcommand\numberthis{\addtocounter{equation}{1}\tag{\theequation}}

\newcommand{\beq}{\begin{equation}}
\newcommand{\enq}{\end{equation}}

\makeatletter
\def\theequation{\@arabic\c@equation}

\newcommand{\bbR}{{\mathbb{R}}}

\newcommand{\bbZ}{{\mathbb{Z}}}
\newcommand{\bbC}{{\mathbb{C}}}

\newcommand{\bbT}{{\mathbb{T}}}

\newcommand{\cB}{{\mathcal B}}

\newcommand{\lb}{\label}

\newcommand{\bu}{{\mathbf u}}
\newcommand{\bv}{{\mathbf v}}
\newcommand{\bk}{{\mathbf k}}

\newcommand{\bH}{{\mathbf H}}

\newcommand{\spec}{\sigma}

\newcommand{\bi}{\bibitem}

\numberwithin{equation}{section}

\renewcommand{\div}{\operatorname{div}}
\renewcommand{\det}{\operatorname{det}}

\newcommand{\dom}{\operatorname{dom}}
\newcommand{\nt}{{\nabla^\perp}}
\DeclareMathOperator{\Tr}{Tr}

\newcommand{\curl}{\operatorname{curl}}

\renewcommand{\Re}{\operatorname{Re }}

\renewcommand{\ker}{\operatorname{ker}}
\newcommand{\diag}{\operatorname{diag}}

\theoremstyle{plain}
\newtheorem{theorem}{Theorem}[section]
\newtheorem{hypothesis}[theorem]{Hypothesis}

\newtheorem{lemma}[theorem]{Lemma}

\newtheorem{proposition}[theorem]{Proposition}

\theoremstyle{definition}
\newtheorem{definition}[theorem]{Definition}

\newtheorem{example}[theorem]{Example}

\newtheorem{remark}[theorem]{Remark}

\begin{document}
\allowdisplaybreaks

\title[2D Euler eigenvalues via Birman-Schwinger and Lin's operators]{Eigenvalues of the linearized 2D Euler equations via Birman-Schwinger and Lin's operators}
\thanks{Partially supported by  NSF grant DMS-171098, Research Council of the University of Missouri and the Simons Foundation.}
\author[Y. Latushkin]{Yuri Latushkin}
\address{Department of Mathematics,
University of Missouri, Columbia, MO 65211, USA}
\email{latushkiny@missouri.edu}
\urladdr{http://www.math.missouri.edu/personnel/faculty/latushkiny.html}
\author[S. Vasudevan]{Shibi Vasudevan}
\address{International Centre for Theoretical Sciences,
Tata Institute of Fundamental Research, Bengaluru, 560089, India}
\email{shibi.vasudevan@icts.res.in}
\date{\today}


\begin{abstract}
We study spectral instability of steady states to the linearized 2D Euler equations on the torus written in vorticity form via certain Birman-Schwinger type operators $K_{\lambda}(\mu)$ and their associated 2-modified perturbation determinants $\mathcal D(\lambda,\mu)$. Our main result characterizes the existence of an unstable eigenvalue to the linearized vorticity operator $L_{\rm vor}$ in terms of zeros of the 2-modified Fredholm determinant $\mathcal D(\lambda,0)=\det_{2}(I-K_{\lambda}(0))$ associated with the Hilbert Schmidt operator $K_{\lambda}(\mu)$ for $\mu=0$. As a consequence, we are also able to provide an alternative proof to an instability theorem first proved by Zhiwu Lin which relates existence of an unstable eigenvalue for $L_{\rm vor}$ to the number of negative eigenvalues of a limiting elliptic dispersion operator $A_{0}$.
\end{abstract}

\maketitle
\normalsize

\section{Introduction} \lb{s1}
The problem of finding unstable eigenvalues for the differential operator obtained by linearizing the two dimensional Euler equations of ideal fluid dynamics about a steady state is classical \cite{DZ04,F71,FS01, FS05, FY99, S00}. Besides the already mentioned general sources, we cite \cite{BFY99,BN10,GC14, Fr, FH98, FSV97, FVY00, ZL03, ZL04} and emphasize that our list is drastically incomplete. The main objective of the current note is to involve perturbation determinants into the study of spectral problems for the linearized Euler equations. The original inspiration for this paper comes from the work by Zhiwu Lin \cite{ZL01,ZL03,ZL04}. We believe that we streamlined and somewhat clarified his approach. Another important predecessor of this paper is \cite{GC14} where a direct integral decomposition of the linearized Euler operator was obtained. 

We realize that in the current paper we obtain a rather theoretical result as we do not have a striking new example of instability proved via perturbation determinants (see, however, Section 5 for an example). Our main achievement is a characterization of isolated eigenvalues of the linearized operator as zeros of certain analytic function of the spectral parameter, the two-modified Fredholm determinant, which in turn leads to a new proof of a theorem by Zhiwu Lin \cite{ZL04}. We do not know if the perturbation determinants were previously used for the linearized Euler operators, and we supply herein this seemingly missing tool.

We consider the two dimensional Euler equations in vorticity form
\begin{equation*}
\omega_t+\bu\cdot\nabla\omega=0
\end{equation*}
on the two torus $\bbT^{2}$, where $\omega=\curl \bu$. Let  $\bu^{0}\cdot \nabla \omega^{0}=0$ be a steady state solution where $\omega^{0}=-\Delta \psi^{0}$ and $\omega^{0}=g(\psi^{0})$ for a sufficiently smooth function $g:\bbR \to \bbR$. Here, $\psi^{0}$ is the stream function associated with the steady state $\omega^{0}$ and $\bu^{0}$, that is, $\bu^{0}=(\psi^{0}_{y},\psi^{0}_{x})$. Linearizing about this steady state, we obtain the equation
\[\omega_t+\bu^0\cdot\nabla\omega+\curl^{-1}\omega\cdot\nabla\omega^0=0.\] 
We define the linearized vorticity operator $L_{\rm vor}$ by
\[ L_{\rm vor}\omega=-\bu^0\cdot\nabla\omega-\curl^{-1}\omega\cdot\nabla\omega^0.\] In this paper, we study the discrete spectrum of the operator $L_{\rm vor}$.
Our main tools are Birman-Schwinger type operators, $K_{\lambda}(\mu)$ (discussed in more detail in Section 4 and defined in equation \eqref{klmudef}), and Lin's operators, $A_{\lambda}$. The Lin's dispersion operators were introduced and studied by Zhiwu Lin in \cite{ZL01,ZL03,ZL04} and defined by the formulae 
\begin{align*}
&A_\lambda =-\Delta-g'(\psi^0(x,y))+g'(\psi^0(x,y))\lambda(\lambda-L^0)^{-1}, \quad \lambda > 0, \\
&A_0=-\Delta-g'(\psi^0(x,y))+g'(\psi^0(x,y))P_0, \quad \lambda =0,
\end{align*}
where $\psi^{0}$ is the stream function for the steady state, $g$ is the real function relating the vorticity and the stream function via $\omega^{0}=g(\psi^{0})$, $L^{0}$ is the operator of differentiation along streamlines given by the formula $L^{0}f=-\bu^{0}\cdot \nabla f$, and $P_{0}$ is the orthogonal projection onto the kernel of $L^{0}$.

A remarkable property of the dispersion operators discovered by Z.\ Lin is that $\lambda > 0$ is an eigenvalue of the operator $L_{\rm vor}$ if and only if $0$ is an eigenvalue of $A_{\lambda}$; cf. Proposition \ref{LvorAlambda}. With this fact in mind, we introduce a family of Birman-Schwinger operators, $K_{\lambda}(\mu)$, which belong to the ideal $\mathcal B_{2}$ of Hibert-Schmidt operators and satisfy the identity $A_{\lambda}-\mu=(I-K_{\lambda}(\mu))(-\Delta-\mu))$, and define the 2-modified Fredholm determinants $\mathcal D(\lambda,\mu)=\det_{2}(I-K_{\lambda}(\mu))$, see formulas \eqref{klmudef} and \eqref{new455} below. As a result, in this paper, we describe the unstable eigenvalues of $L_{\rm vor}$ as zeros of the analytic function $\mathcal D(\cdot,0)$ for $\mu=0$, see Theorem \ref{newequivalence}, which is the main new result of this paper. In addition, we give a new version of the proof of an important theorem by Z.\ Lin saying that if $A_{0}$ has no kernel and an odd number of negative eigenvalues, then the operator $L_{\rm vor}$ has at least one positive eigenvalue, see Theorem \ref{LinThm} below. The proof is based on the fact that $A_{\lambda}-A_{0}$ converges to zero strongly in $L^{2}(\mathbb T^{2})$ and, as a result, that $K_{\lambda}(\mu)-K_{0}(\mu)$ converges to zero in $\mathcal B_{2}$ as $\lambda \to 0^{+}$. Due to the convergence of the respective perturbation determinants, it follows that the number of negative eigenvalues of $A_{0}$ and $A_{\lambda}$ for small $\lambda > 0$ (which is equal to the number of zeros of $\mathcal D(0,\cdot)$ and $\mathcal D(\lambda,\cdot)$, respectively) coincide, see Proposition \ref{A0Alambd} below. Since $A_{\lambda}$ has no negative eigenvalues for large $\lambda >0$, and has an even number of nonreal eigenvalues, this shows that when $\lambda$ changes from small to large positive values, an eigenvalue of $A_{\lambda}$ must cross through zero, thus proving the existence of a positive eigenvalue of $L_{\rm vor}$.

\section{Problem setup and preliminaries} %
Consider the two dimensional inviscid Euler equations
\begin{equation}\lb{4EE}
\bu_t+\big(\bu\cdot\nabla\big)\bu+\nabla p=0,\quad \div\bu=0,
\end{equation}
on the torus $\bbT^2=\bbR^2/2\pi\bbZ^2$. Since $\div \bu=0$, there is a stream function $\psi$ such that equation $\bu=-\nt\psi$ holds, where $-\nt$ is the vector $(-\partial_{y},\partial_{x})$ so that $-\nt \psi =(\psi_{y},\psi_{x})$. We introduce the vorticity $\omega=\curl\bu$ so that $\omega=-\Delta\psi$.
 
Applying $\curl$ in \eqref{4EE}, one obtains the Euler equation in vorticity form,
\begin{equation}\lb{4EEV}
\omega_t+\bu\cdot\nabla\omega=0.
\end{equation}
The Euler equation for the stream function $\psi$ is
\begin{equation}\lb{4EES}
\Delta\psi_t-\nt\psi\cdot\nabla(\Delta\psi)=\Delta\psi_t-\psi_x\Delta\psi_y+\psi_y\Delta\psi_x=0.
\end{equation}

Let us consider a smooth steady state solution $\omega^0=\curl\bu^0=-\Delta\psi^0$ of \eqref{4EEV}. In particular,
$\nt\psi^0\cdot\nabla\omega^0=\big(-\psi^0_y\partial_x+\psi^0_x\partial_y\big)\omega^0=0$, and thus $\nabla\psi^0$ and $\nabla(\Delta\psi^0)$ are parallel. Assume furthermore
\begin{hypothesis}\label{4hyp}
There exists a smooth function $g:\mathbb R \to \mathbb R$, such that the equation
\begin{equation}\lb{ompsi0}
\omega^0(x,y)=-\Delta\psi^0(x,y)=g(\psi^0(x,y))
\end{equation}
holds for all $(x,y) \in \mathbb T^{2}$. 
\end{hypothesis}
Hypothesis \ref{4hyp} in turn, implies that 
\begin{equation}\lb{ompsi1}
\nt\omega^0=g'(\psi^0)\nt\psi^0.
\end{equation}
We linearize the Euler equations \eqref{4EE}--\eqref{4EES} about the steady state:
\begin{align}
\bu_t&+\bu^0\cdot\nabla\bu+\bu\cdot\nabla\bu^0+\nabla p=0,\quad \div\bu=0, \lb{4LEE}\\
\omega_t&+\bu^0\cdot\nabla\omega+\curl^{-1}\omega\cdot\nabla\omega^0=0, \lb{4LEEV}\\
\Delta\psi_t&-\psi^0_x\Delta\psi_y+\psi^0_y\Delta\psi_x-\psi_x\Delta\psi^0_y+\psi_y\Delta\psi^0_x=0.\lb{4LEES}
\end{align}
Here, $\bu=\curl^{-1}\omega$ denotes the unique solution of the system
$\curl\bu=\omega$, $\div\bu=0$ with $\omega$ having zero space average $\int_{\bbT^{2}}\omega\,dx\,dy=0$.

We introduce the respective linear operators $L_{\rm vel}$, $L_{\rm vor}$, $L_{\rm str}$ corresponding to \eqref{4LEE}--\eqref{4LEES}, on the following Sobolev spaces. %
We fix an $m\in\bbZ$, and denote by $H_a^m$ the Sobolev space of (scalar $W^m_2(\bbT^2)$) functions or distributions with zero space average:
\begin{equation}
H_a^m=\Big\{w(x,y)=\sum_{\bk\in\bbZ^2}w_\bk e^{i\bk\cdot(x,y)}\,\Big|\,
w_{\bf 0}=0, \sum_{\bk\in\bbZ^2}(1+|\bk|^{2})^{m}|w_\bk|^2<\infty \Big\},
\end{equation}
and by ${\mathbf H}^m_s$ the Sobolev space of (solenoidal vector valued) functions with zero divergence and zero space average:
\begin{equation}
{\mathbf H}^m_s=\Big\{\bv(x,y)=\sum_{\bk\in\bbZ^2}\bv_\bk e^{i\bk\cdot(x,y)}\,\Big|\, \bv_0=0,
\div\bv_{\bk}=0, \sum_{\bk\in\bbZ^2}(1+|\bk|^{2})^{m}\|\bv_\bk\|^2<\infty \Big\}.
\end{equation}
Setting
\begin{align}
L_{\rm vel}\bu&=-\bu^0\cdot\nabla\bu-\bu\cdot\nabla\bu^0-\nabla p,\lb{Lvel}\\
L_{\rm vor}\omega&=-\bu^0\cdot\nabla\omega-\curl^{-1}\omega\cdot\nabla\omega^0,\lb{Lvor}\\
L_{\rm str}\psi&=-\Delta^{-1}\big(-\psi^0_x\,\Delta\psi_y+\psi^0_y\,\Delta\psi_x-\psi_x\,\Delta\psi^0_y+\psi_y\,\Delta\psi^0_x\big)\nonumber\\
&=\Delta^{-1}\big(-\bu^0\cdot\nabla(\Delta\psi)+\nt\psi\cdot\nabla(\Delta\psi^0)\big),\lb{Lstr}\\
&=\Delta^{-1}\big(-\bu^0\cdot(\Delta\psi)-\curl^{-1}(\Delta\psi)\cdot\nabla(-\Delta\psi^0)\big),\nonumber
\end{align}

we observe that the following diagrams commute:
\begin{equation}\lb{dia}\begin{CD}
H_a^{m+2}@>L_{\rm str}>> H_a^{m+1}\\
 @VV\nt V @ A (\nt)^{-1}AA\\
\bH^{m+1}_s@>L_{\rm vel}>> \bH^{m}_s\\
@VV\curl V @ A (\curl)^{-1}AA\\
H_a^{m}@>L_{\rm vor}>> H_a^{m-1}
\end{CD}, \qquad 
\begin{CD} H_a^{m+2}@>L_{\rm str}>> H_a^{m+1}\\
 @V\Delta VV @ AA (\Delta)^{-1}A\\
 H_a^{m}@>L_{\rm vor}>> H_a^{m-1}.
\end{CD}\end{equation}

\begin{remark}\label{steadystateassumptions}
We need a few further assumptions about the steady state. We follow the notation in \cite{GC14}. Denote by $\widehat{D}$, the union of the images of periodic orbits of the flow generated by $\bu^{0}$ and by $D_{0}=\{(x,y):\nabla^{\perp}\psi^{0}(x,y)=0\}$ the set of fixed points of the flow. We further assume that the periodic orbits together with the fixed points ``fill up'' the torus, i.e., more precisely, we have the following hypothesis.
\begin{hypothesis}\label{periodicorbitsassump}
In addition to Hypothesis \ref{4hyp} we assume that $\mathbb T^{2} \backslash (\widehat{D}\cup D_{0})$ has measure zero in $\mathbb T^{2}$.
\end{hypothesis}
For any $\rho$ in the image of $\psi^{0}$ which is not a critical value, the level sets $\{(x,y):\psi^{0}(x,y)=\rho\}$ consist of a finite number of disjoint closed curves which we denote by $\Gamma_{1}(\rho), \Gamma_{2}(\rho), \ldots, \Gamma_{n(\rho)}(\rho)$. Define the set $\mathcal J$ to be the disjoint union of the values assumed by the steady state $\psi^{0}$ on each periodic orbit, i.e., as in \cite{GC14}, we set
\begin{equation}\label{jdef}
\mathcal{J} :=  \amalg_{i} \psi^{0}(x,y)|_{\Gamma_{i}(\rho)}.
\end{equation}
This set assumes the values of the stream function on periodic orbits, counted with multiplicity. It is a disjoint union of open intervals. For example, if we consider the steady state $\bu^{0}=(\cos y, 0)$ on $\mathbb T^{2}$, then $\mathcal{J}=(-1,1) \amalg (-1,1)$. One can thus define the period function $T:\mathcal J \to (0,\infty)$ by letting $T(\rho)$ be the time period of the orbit corresponding to $\rho \in \mathcal J$. Henceforth, we denote by $\Gamma(\rho)$ the periodic orbit corresponding to $\rho \in \mathcal J$.  
\end{remark}

\section{Lin's dispersion operators and their properties}\label{Linsop}
In this section we define and study elliptic dispersion operators, $A_{\lambda}$, parametrized by the spectral parameter $\lambda \geq 0$ for the linearized Euler operator $L_{\rm vor}$ defined in \eqref{Lvor}. The operators $A_{\lambda}$ have the remarkable property that $\lambda >0$ is an isolated eigenvalue of the linearized Euler operator $L_{\rm vor}$ if and only if $0$ is in the spectrum of $A_{\lambda}$. The operators $A_{\lambda}$ are often called dispersion operators; they were introduced by Zhiwu Lin in \cite{ZL01,ZL03,ZL04} and studied in subsequent work, see, e.g. \cite{GGVR07, S08}. We call them Lin's operators. 

Setting
$m=1$ in the diagrams \eqref{dia}, we will discuss the spectrum of the operator $L_{\rm vor}$ in $L^{2}_{a}$. Thus, we will view the operator $L_{\rm vor}$ in \eqref{Lvor} as an unbounded first order differential operator on $L^{2}_{a}$ with the domain $H_a^{1}$, the Sobolev space of functions with zero space average.
If $\omega$ is an eigenfunction of $L_{\rm vor}$ corresponding to the eigenvalue $\lambda$ with $\Re(\lambda)>0$, then $\omega\in H_a^1$ and $L_{\rm vor}\omega=\lambda\omega$ show that the
linearized Euler equation \eqref{4LEEV} has the solution $e^{\lambda t}\omega$ whose $L^2(\bbT^{2})$-norm grows exponentially. Respectively, the enstrophy, that is, $H^2(\bbT^{2})$-norm of the corresponding stream function $e^{\lambda t}\psi$, such that $\omega=-\Delta\psi$,
grows exponentially.

Let us denote by $\varphi^t$ the flow on $\bbT^{2}$ generated by the vector field $\nt\psi^0=-\bu^0$. In other words, $\varphi^t(x,y)=(X(t;x,y),Y(t;x,y))$ where $X$ and $Y$ solve the Cauchy problem 
\begin{equation}
X_t=-\psi^0_y(X,Y),\, Y_t=\psi^0_x(X,Y),\, X(0)=x, \, Y(0)=y.
\end{equation}
The strongly continuous evolution group $\{T^t\}_{t\in\bbR}$ on $L^2(\bbT^{2})$ of unitary operators defined by $T^t\phi=\phi\circ\varphi^t$ is generated by the first order differential operator $L^{0}$ defined by
\begin{equation}
L^0\phi=\nt\psi^0\cdot\nabla\phi=\big(-\psi^0_y\partial_x+\psi^0_x\partial_y\big)\phi,
\end{equation} 
with the domain
$\dom L^0=H^1_a$. We have the following properties of  $L^{0}$.

\begin{proposition}
\begin{align}\lb{prop1}
&(i) \quad \big(L^0\big)^*=-L^0;\\
&(ii) \quad \spec(L^0)\subseteq i\bbR ;\lb{prop2}\\
&(iii) \quad \spec(L^0)= i\bbR\quad %
\text{ provided $\varphi^{t}$ has arbitrary long orbits};\lb{prop3n} \\
&(iv) \quad \|(\lambda-L^0)^{-1}\|_{\cB(L_2)} \leq |\Re(\lambda)|^{-1},\, \Re(\lambda)\neq0;\lb{prop3}\\
&(v) \quad \|(\lambda-L^0)^{-1}\|_{\cB(L_2)} = |\Re(\lambda)|^{-1},\, \Re(\lambda)\neq0;\nonumber \\ & 
\text{ provided $\varphi^{t}$ has arbitrary long orbits;}\lb{prop3a}\\
&(vi) \quad L^0 \text{ and } (\lambda-L^0)^{-1}\text{ are normal operators};\lb{prop4}\\
&(vii) \quad L^0 \text{ and the operator of multiplication by } g'(\psi^0(x,y)) \text{ commute. }\lb{prop5}
\end{align}
\end{proposition}
\begin{proof}
Property \eqref{prop1} follows by integrating by parts and implies \eqref{prop2}. Property \eqref{prop3n} is proved, for example, in \cite[Theorem 5]{SL05} or \cite{CL99}. Properties \eqref{prop3} and \eqref{prop3a} are proved by passing to the self adjoint operator $iL^{0}$. Property \eqref{prop4} is obvious. Property \eqref{prop5}, see \eqref{ompsi0}, follows from the fact that $\psi^0(\varphi^t(x,y))$ is $t$-independent due to $\nabla\psi^0\cdot\nt\psi^0=0$.
\end{proof}

We will now show that $\lambda\in\bbC\setminus i\bbR$ is an eigenvalue of the operator $L_{\rm vor}$ acting in $L_2(\bbT^{2})$ if and only if $0$ is an eigenvalue of a certain elliptic operator, $A_\lambda$ in $L_2(\bbT^{2})$, having a sufficiently smooth eigenfunction. 
\begin{definition}\label{dfnAl}
We define $A_\lambda$ with the domain $\dom(A_\lambda)=\dom\Delta=H^{2}_{a}$ for nonimaginary $\lambda$ as follows:
\begin{align}
A_\lambda&=-\Delta+g'(\psi^0(x,y))L^0(\lambda-L^0)^{-1}\lb{defAl1}\\
&=-\Delta-g'(\psi^0(x,y))+g'(\psi^0(x,y))\lambda(\lambda-L^0)^{-1},
\, \Re(\lambda)\neq0.\lb{defAl2}
\end{align}
\end{definition}
We remark that the operator $A_\lambda$ is not self-adjoint; it is a perturbation of a self-adjoint operator (either $-\Delta$, see \eqref{defAl1},
or $-\Delta-g'(\psi^0)$, see \eqref{defAl2}) by a normal, see \eqref{prop4}, \eqref{prop5}, relatively compact bounded operator (either $g'(\psi^0)L^0(\lambda-L^0)^{-1}$ or $\lambda g'(\psi^0)(\lambda-L^0)^{-1}$).

\begin{remark}\lb{comcoj}
Since $A_\lambda$ commutes with complex conjugation, the nonreal eigenvalues of $A_\lambda$ are complex conjugate. Thus, the number of the nonreal eigenvalues of $A_{\lambda}$ must be even for each value of $\lambda$.\end{remark}
The following calculation shows how the eigenfunctions of $L_{\rm vor}$ are related to the elements of the kernel of $A_{\lambda}$.
Using $\bu=\curl^{-1}\omega=-\nt\psi$ and \eqref{ompsi1}, let us re-write
\eqref{Lvor} as follows:
\begin{align}
L_{\rm vor}\omega&=\nt\psi^0\cdot\nabla\omega+\nt\psi\cdot\nabla\omega^0=\nt\psi^0\cdot\nabla\omega-\nt\omega^0\cdot\nabla\psi\\
&=\nt\psi^0\cdot\nabla\omega-g'(\psi^0(x,y))\nt\psi^0\cdot\nabla\psi\\
&=L^0\omega-g'(\psi^0(x,y))L^0\psi.\label{lvorloreln}
\end{align}
Thus, 
\begin{equation}\label{LvoriffA}
L_{\rm vor}\omega=\lambda\omega\text{ if and only if }
(\lambda-L^0)\omega=-g'(\psi^0(x,y))L^0\psi,\, \lambda\in\bbC,
\end{equation}
where $\omega$ and $\psi$ are related via $\omega=-\Delta\psi$.
Using \eqref{prop2} and \eqref{prop5}, we see that 
\begin{equation}\label{Lvoriff}
L_{\rm vor}\omega=\lambda\omega\text{ if and only if }
\omega=-g'(\psi^0(x,y))L^0(\lambda-L^0)^{-1}\psi,\, \Re (\lambda)\neq0.
\end{equation}
Using $\omega=-\Delta\psi$, the second equation in \eqref{Lvoriff} can be re-written as
\begin{equation}\label{LvoriffB}
A_\lambda\psi=-\Delta\psi+g'(\psi^0(x,y))L^0(\lambda-L^0)^{-1}\psi=0.
\end{equation}
We thus have the following fact first proved by a slightly different argument in \cite[Lemma 3.2]{ZL04}.
\begin{proposition}\label{LvorAlambda}
A non imaginary $\lambda$ belongs to $\sigma_{p}(L_{\rm vor})$ if and only if $0$ belongs to $\sigma_{p}(A_{\lambda})$. Specifically, if $\omega\in\dom(L_{\rm vor})=H_a^1(\bbT^{2})$ is an eigenfunction of the operator $L_{\rm vor}$ and $\Re(\lambda)\neq0$ then $\psi=-\Delta^{-1}\omega\in H_a^3\subset\dom(A_\lambda)$ satisfies $A_\lambda\psi=0$. Conversely, if $\psi\in\dom(A_\lambda)$ satisfies $A_\lambda\psi=0$ with $\Re(\lambda)\neq0$ then $\omega=-\Delta\psi\in\dom(L_{\rm vor})=H_a^1(\bbT^{2})$ is an eigenfunction of $L_{\rm vor}$.
\end{proposition}
\begin{proof}
By \eqref{LvoriffA}-\eqref{LvoriffB}, if $\omega$ is an eigenfunction of $L_{\rm vor}$ then $\psi=-\Delta^{-1}\omega$ is an eigenfunction of $A_{\lambda}$. Conversely, assuming $\psi \in \ker(A_{\lambda})$, we have that $\psi \in H_{a}^{1}=\dom A_{\lambda}$ satisfies the elliptic equation $-\Delta\psi-g'(\psi^{0})\psi=\phi$, where we temporarily define $\phi=-(\lambda-L^0)^{-1}\lambda g'(\psi^0)\psi \in H_{a}^{1}$. Due to elliptic regularity the solution $\psi$ of the equation is in fact in $H_{a}^{1+2}=H_{a}^{3}$, see also \cite{ZL04a}. Then $\omega=-\Delta\psi \in H_{a}^{1}=\dom(L_{\rm vor})$. Running \eqref{LvoriffA}-\eqref{LvoriffB} backwards, we have that $\omega$ is an eigenfunction of $L_{\rm vor}$.
\end{proof}
We now discuss what happens to the operators $A_\lambda$ when $|\Re (\lambda)|$ is large, and when $\lambda\to0^+$. We begin with large $|\Re(\lambda)|$. %
Since  
\begin{equation}\label{L0est}
\|L^{0}\psi\|_{L^2}=\|\nabla^{\perp}\psi^{0}\cdot\nabla\psi\|_{L^2} \leq \|\nabla\psi^{0}\|_{L^{\infty}}\|\nabla\psi\|_{L^2} \leq c \|\nabla\psi\|_{L^2},
\end{equation}
we infer that there is a constant $c$ such that for all $\psi\in H_{a}^1$ and $\Re(\lambda)\neq0$ one has:
\begin{equation}\label{estlarge}
\begin{split}
\Big|\big\langle &  g'(\psi^0)L^0(\lambda-L^0)^{-1}\psi,\psi\big\rangle_{L^2}\Big|\\& \le\|g'(\psi^0)\|_{L_\infty}\|(\lambda-L^0)^{-1}\|_{\cB(L^2)}\|\psi\|_{L^2}\|L^0\psi\|_{L^2}\\
&\le c\|g'(\psi^0)\|_{L_\infty}|\Re(\lambda)|^{-1}\|\psi\|_{L^2}\|\nabla\psi\|_{L^2}
\quad\text{(by \eqref{prop3} and \eqref{L0est}})\\
&\le c_0\|g'(\psi^0)\|_{L_\infty}|\Re(\lambda)|^{-1}\|\nabla\psi\|_{L^2}^2
\quad\text{(by the Poincare inequality}).
\end{split}\end{equation}
Thus, for all $\psi\in H_{a}^2$ and sufficiently large $|\Re(\lambda)|$, the
Poincare inequality yields
\begin{align}
\Re\big\langle A_\lambda\psi,\psi\big\rangle_{L^2}&\ge
\|\nabla\psi\|_{L^2}-c_0\|g'(\psi^0)\|_{L_\infty}|\Re(\lambda)|^{-1}\|\nabla\psi\|_{L^2}^2 \nonumber \\&\ge c_1\|\psi\|_{L^2}^2,
\end{align}
and thus $A_\lambda$ has no eigenvalues with negative real parts provided $|\Re(\lambda)|$ is large enough, cf. \cite[Lemma 3.4]{ZL04}. In particular, there exists $\lambda_\infty>0$ such that 
\begin{equation}\lb{noneg}
\text{if $\lambda\ge\lambda_\infty$ then $A_\lambda$ has no negative eigenvalues.}\end{equation}

We will now discuss what happens to the operator $A_\lambda$ when  $\lambda\to0^+$. First, let us consider the operator $\lambda(\lambda-L^0)^{-1}$ from \eqref{defAl2}. To motivate the limiting behavior of this operator as $\lambda \to 0^{+}$, let us impose, for a second, an additional assumption that zero is an isolated eigenvalue of $L^0$ (this is indeed a strong assumption that holds provided all orbits of $\varphi^t$ are periodic uniformly with bounded periods). Then the usual expansion of the resolvent operator, see \cite[Section III.6.5]{Ka80}, around zero, $(\lambda-L^0)^{-1}=P_0\lambda^{-1}+D_0+D_1\lambda+\dots$, yields that $P_0=\lim_{\lambda\to0}\lambda(\lambda-L^0)^{-1}$ is the Riesz spectral projection for $L^0$ onto $\ker L^0$. A remarkable fact proved in \cite[Lemma 3.5]{ZL04} (in a different form and using a different method) is that this limiting relation holds without any additional assumptions but in the sense of strong convergence. In particular, going back to the general case, we have the following fact first proved by Z.\ Lin in \cite[Lemma 3.5]{ZL04}.
\begin{lemma}\lb{lem_proj}
Assume Hypothesis \ref{periodicorbitsassump} and let $P_0$ denote the orthogonal projection in $L^{2}_{a}$ onto the subspace  $\ker L^0=\{\phi\in H_a^1:
\nt\psi^0\cdot\nabla\phi=0\}$. Then for any $\phi\in L_a^2$ one has
$\lambda(\lambda-L^0)^{-1}\phi\to P_0\phi$ in $L^2$ as $\lambda\to0^+$.
\end{lemma}
\begin{proof}
We recall Remark \ref{steadystateassumptions} about the assumptions regarding the steady state. Hypothesis \ref{periodicorbitsassump} implies that it is enough to check the lemma on $D_{0}\cup \widehat{D}$. On $D_{0}$, the operator $L^{0}|_{L^{2}(D_{0})}$ is the zero operator, and the orthogonal projection $P_{0}|_{L^{2}(D_{0})}$ is the identity operator and the statement of the lemma holds trivially. It is thus enough to check the lemma on $\widehat{D}$. Following \cite{ZL04} and using co-area formula, see \cite[Formula 16]{GC14}, the $L^{2}(\widehat{D})$ norm of any function $\phi \in L^{2}(D)$ restricted to the set $\widehat{D}$ can be represented as
\begin{equation}\label{coareal2}
\|\phi\|^{2}_{L^{2}(\widehat{D})}= \int_{-\infty}^{\infty} \bigg(\int_{(\psi^{0})^{-1}(\rho)} \frac{|\phi|^{2}}{|\nabla \psi^{0}|}ds\bigg) d\rho=\int_{\mathcal{J}} \bigg(\int_{\Gamma(\rho)} \frac{|\phi|^{2}}{|\nabla \psi^{0}|}ds\bigg) d\rho.
\end{equation}
Here, $ds$ is the induced measure on each streamline $\Gamma(\rho)$, and $d\rho$ is the Lebesgue measure on the index set $\mathcal{J}$ defined in \eqref{jdef}. By assumption, each $\Gamma(\rho)$ is diffeomorphic to the unit circle $\mathbb T$ with period $T(\rho)$. As time $t$ varies from $0$ to $T(\rho)$, a point $\varphi^{t}(x,y)$ on $\Gamma(\rho)$ traces one full orbit around $\Gamma(\rho)$. One can see that $dt= ds / |\nabla \psi^{0}|$ because $d\varphi^{t}/ dt = -\nabla^{\perp}\psi^{0}(\varphi^{t})$. Thus, one can rewrite \eqref{coareal2} as
\begin{equation}\label{coareal2t}
\|\phi\|^{2}_{L^{2}(\widehat{D})}= \int_{\mathcal{J}} \bigg(\int_{\Gamma(\rho)} |\phi_{\rho}(t)|^{2}dt \bigg) d\rho,
\end{equation}
where the restriction of $\phi$ to $\Gamma(\rho)$ is denoted by $\phi_{\rho}$ which is in the space $ L^{2}_{per}(0,T(\rho))$, i.e., the space of  $L^{2}$ functions on $(0,T(\rho))$ with periodic boundary conditions. In the language of direct integral decomposition of operators, see for example \cite[Section XIII.16]{RS78}, we can represent the space $L^{2}(\widehat{D})$ as 
\begin{equation}\label{directintfuncspace}
L^{2}(\widehat{D})=\int_{\mathcal J}^{\oplus} L^{2}_{per}(0,T(\rho))d \rho.
\end{equation}
We will use the direct integral decomposition of $L^{0}$ from \cite{GC14}. Fix a point $(x,y) \in \Gamma(\rho)$. We first note that $L^{0}|_{\Gamma(\rho)}f(x,y) = -\frac{d}{dt}|_{t=0} f(\varphi^{t}(x,y))$, where $\varphi^{t}(x,y)$ is the flow generated by the velocity field $\bu^{0}$ and $f$ is smooth. Thus, if $\phi_{\rho} \in L^{2}_{per}(0,T(\rho))$ with the Fourier series representation $\phi_{\rho}(t)=\sum_{k \in \bbZ} \widehat{\phi}_\rho(k)e^{2\pi ikt/T(\rho)}$, where $t \in [0,T(\rho))$, then 
\begin{align}\label{l0fourier}
L^{0}(\rho)\phi_{\rho}(t)&= -\frac{d}{dt}\phi_{\rho}(t) =-\frac{d}{dt} \sum_{k \in \bbZ} \widehat{\phi}_\rho(k)e^{2\pi ikt/T(\rho)} \nonumber \\&= - \sum_{k \in \bbZ} \frac{2 \pi ik}{T(\rho)} \widehat{\phi}_\rho(k)e^{2\pi ikt/T(\rho)},
\end{align}
and where $L^{0}(\rho)$ represents the restriction of $L^{0}$ to $L^{2}(\Gamma(\rho))$. Using \eqref{l0fourier} one can obtain a representation for the resolvent operator as
\begin{align*}
 &(\lambda - L^{0}(\rho))^{-1} \bigg(\sum_{k \in \bbZ} \widehat{\phi}_\rho(k)e^{2\pi ikt/T(\rho)}\bigg)= \sum_{k \in \bbZ} \frac{1}{\lambda+\frac{2\pi ik}{T(\rho)}}\widehat{\phi}_\rho(k) e^{2\pi ikt/T(\rho)}\\&=\sum_{k \in \bbZ} \frac{T(\rho)}{\lambda T(\rho) +2\pi ik}\widehat{\phi}_\rho(k) e^{2\pi ikt/T(\rho)}. \numberthis \label{resolventformula}
\end{align*}
We thus have, for each Fourier coefficient $\widehat{\phi}_\rho(k)$,
\begin{equation}\label{0limit}
 \lim_{\lambda \to 0^+}\lambda(\lambda - L^{0}(\rho))^{-1} \widehat{\phi}_\rho(k)= \lim_{\lambda \to 0^+} \frac{\lambda T(\rho)}{\lambda T(\rho) +2\pi ik}\widehat{\phi}_\rho(k)=\begin{cases}
   \widehat{\phi}_\rho(0) & \text{if } k = 0, \\
   0       & \text{if } k \neq 0.
  \end{cases}
\end{equation}
Let us denote by $P_{0}(\rho)$ the orthogonal projection in $L^{2}(\Gamma(\rho))$ onto the kernel of $L^{0}(\rho)$. We note here that in the language of direct integral decomposition of spaces and operators, see \cite[Section XIII.16]{RS78}, we can write, as proved in \cite{GC14}, that 
\begin{equation}\label{directintl0p0}
P_{0}= \int_{\mathcal J}^{\oplus} P_{0}(\rho)d \rho, \quad iL^{0}= \int_{\mathcal J}^{\oplus} iL^{0}(\rho)d \rho. 
\end{equation}
From \eqref{l0fourier}, we see that $L^{0}(\rho)\phi_{\rho}=0$ if and only if $\widehat{\phi}_\rho(k)=0$ for every $k \neq 0$. Thus we have that 
\begin{equation}\label{P0formula}
P_{0}(\rho)\phi_{\rho}=\widehat{\phi}_\rho(0).
\end{equation}
We now claim that if $\phi \in L^{2}(\widehat{D})$ and $\phi_{\rho} \in L^{2}(\Gamma(\rho))$ then for each $\rho$ one has
\begin{equation}\label{orbitconv}
\lim_{\lambda \to 0 ^{+}} \| (\lambda(\lambda-L^{0}(\rho))^{-1}-P_{0}(\rho))\phi_{\rho}\|_{L^{2}(\Gamma(\rho))}^{2} = 0.
\end{equation}
Indeed, using Parseval's theorem and formulae \eqref{P0formula} and \eqref{resolventformula}, we have,
\begin{equation}\label{fouriercoeffformula}
\| (\lambda(\lambda-L^{0}(\rho))^{-1}-P_{0}(\rho))\phi_{\rho}\|_{L^{2}(\Gamma(\rho))}^{2} = \sum_{k \in \bbZ \backslash \{ 0\}} \frac{\lambda^{2}T^{2}(\rho)}{\lambda^{2} T^{2}(\rho) +4\pi^{2} k^{2}}|\widehat{\phi}_\rho(k)|^{2}.
\end{equation}
For every $k \neq 0$, since $\lambda^{2}T^{2}(\rho)\le{\lambda^{2} T^{2}(\rho) +4\pi^{2} k^{2}}$,
we have that, 
\begin{equation}\label{dominatingterm}
\frac{\lambda^{2}T^{2}(\rho)}{\lambda^{2} T^{2}(\rho) +4\pi^{2} k^{2}}|\widehat{\phi}_\rho(k)|^{2} \leq |\widehat{\phi}_\rho(k)|^{2},
\end{equation}
and
\begin{equation}\label{dominatingfunction}
\sum_{k \in \bbZ \backslash \{ 0\}} \frac{\lambda^{2}T^{2}(\rho)}{\lambda^{2} T^{2}(\rho) +4\pi^{2} k^{2}}|\widehat{\phi}_\rho(k)|^{2} \leq \sum_{k \in \bbZ}|\widehat{\phi}_\rho(k)|^{2}.
\end{equation}
Since \eqref{dominatingterm} holds, one can apply Lebesgue dominated convergence theorem on the space $\ell^{2}(\bbZ)$ to conclude that,
\begin{align}
&\lim_{\lambda \to 0+}\sum_{k \in \bbZ \backslash \{ 0\}} \frac{\lambda^{2}T^{2}(\rho)}{\lambda^{2} T^{2}(\rho) +4\pi^{2} k^{2}}|\widehat{\phi}_\rho(k)|^{2} \nonumber \\&= \sum_{k \in \bbZ \backslash \{ 0\}} \lim_{\lambda \to 0+} \frac{\lambda^{2}T^{2}(\rho)}{\lambda^{2} T^{2}(\rho) +4\pi^{2} k^{2}}|\widehat{\phi}_\rho(k)|^{2}=0.
\end{align}
Hence \eqref{orbitconv} holds, as claimed. 
We are now ready to prove the main assertion in the lemma that $\lambda(\lambda-L^0)^{-1}\phi\to P_0\phi$ in $L^2(\widehat{D})$ as $\lambda\to0^+$. We need to show that
\begin{equation}
\lim_{\lambda \to 0^{+}} \|(\lambda(\lambda-L^{0})^{-1}-P_{0})\phi \|^{2}_{L^{2}(\widehat{D})} = 0. 
\end{equation}
Using formula \eqref{coareal2t} this reduces to showing that
\begin{equation}\label{neweq2}
\lim_{\lambda \to 0^{+}} \int_{\mathcal J} \| (\lambda(\lambda-L^{0}(\rho))^{-1}-P_{0}(\rho))\phi_{\rho}\|_{L^{2}(\Gamma(\rho))}^{2} d\rho = 0. 
\end{equation}
Using \eqref{fouriercoeffformula} and \eqref{dominatingfunction} we have that, applying Parseval's theorem twice,
\begin{equation}
\| (\lambda(\lambda-L^{0}(\rho))^{-1}-P_{0}(\rho))\phi_{\rho}\|_{L^{2}(\Gamma(\rho))}^{2} \leq \|\phi_{\rho}\|_{L^{2}(\Gamma(\rho))}^{2}.
\end{equation}
One can thus apply the Lebesgue dominated convergence theorem to the left hand side of \eqref{neweq2} and use \eqref{orbitconv} to conclude that,
\begin{align*}
&\lim_{\lambda \to 0^{+}} \int_{\mathcal J} \| (\lambda(\lambda-L^{0}(\rho))^{-1}-P_{0}(\rho))\phi_{\rho}\|_{L^{2}(\Gamma(\rho))}^{2} d\rho \\&=  \int_{\mathcal J} \lim_{\lambda \to 0^{+}} \| (\lambda(\lambda-L^{0}(\rho))^{-1}-P_{0}(\rho))\phi_{\rho}\|_{L^{2}(\Gamma(\rho))}^{2} d\rho=0,
\end{align*}
finishing the proof of the lemma.
\end{proof}
We now extend the definition of $A_{\lambda}$ in Definition \ref{dfnAl} as follows.
\begin{definition}\label{dfnA0} 
Introduce the operator $A_0$ in $L^2$ with $\dom(A_0)=\dom(\Delta)=H_a^2$ by the formula
\begin{equation}\lb{defA0}
A_0=-\Delta-g'(\psi^0(x,y))+g'(\psi^0(x,y))P_0.
\end{equation}
\end{definition}
By Lemma \ref{lem_proj} we infer that $A_\lambda\phi\to A_0\phi$ in $L^2$ as $\lambda\to0^+$ for each $\phi\in\dom(A_\lambda)$. %

To conclude this section we formulate the following important theorem proven by Zhiwu Lin in \cite{ZL04} and provide an outline of its proof. 
\begin{theorem}\label{LinThm}
Assume Hypothesis \ref{periodicorbitsassump} and consider the dispersion Lin's operator $A_{0}$ defined in Definition \ref{dfnA0}. Assume that $A_{0}$ has an odd number of negative eigenvalues and no kernel. Then $L_{\rm vor}$ has a positive isolated eigenvalue. 
\end{theorem}
This result gives a new instability criterion for the steady state $\psi^{0}$, cf. \cite{ZL03, ZL04}. The proof given in \cite{ZL04} rests on an abstract theorem based on the use of the infinite determinants of the operators $e^{-A_{\lambda}}$. In the next section we will offer an alternative proof of Theorem \ref{LinThm} based on the use of Birman-Schwinger type operators. In addition, these operators, associated with $A_{\lambda}$ and $-\Delta$,  provide us with an analytic function of the spectral parameter whose zeros are exactly the eigenvalues of the operator $A_{\lambda}$. Computing this function at the value $\mu=0$ for the spectral parameter for $A_{\lambda}$, we will obtain a function of the spectral parameter for $L_{\rm vor}$ whose zeros are exactly the isolated eigenvalues of $L_{\rm vor}$.
\begin{remark}\label{STP}
The strategy for the proof of Theorem \ref{LinThm} is as follows. Assuming that $A_{0}$ has an odd number of negative eigenvalues, we show, see Proposition \ref{A0Alambd}, that $A_{\lambda}$ has the same number of negative eigenvalues provided $\lambda > 0$ are small enough. This is the only assertion to be proved to establish the result in Theorem \ref{LinThm} as the rest is easy. Indeed, as $\lambda > 0$ changes from small to large positive values, one of the eigenvalues of $A_{\lambda}$ should cross through zero since $A_{\lambda}$ has no negative eigenvalues for large $\lambda > 0$ by \eqref{noneg} and the eigenvalues of $A_{\lambda}$ may leave the real line only in pairs by Remark \ref{comcoj}. So, for some $\lambda > 0$ we have $0 \in \sigma_{p}(A_{\lambda})$ and thus $\lambda \in \sigma_{p}(L_{\rm vor})$ by Proposition \ref{LvorAlambda}, completing the proof of Theorem \ref{LinThm}.
\end{remark}

\section{Birman-Schwinger type operator associated with Lin's operator} 
In this section we introduce and study the family of Birman-Schwinger type operators $K_{\lambda}(\mu)$, $\mu \in \mathbb C \backslash \sigma(-\Delta)$, $Re(\lambda) \neq 0$, associated with Lin's dispersion operators $A_{\lambda}$, introduced in Section \ref{Linsop} and the negative Laplace operator $-\Delta$. As we will see, $K_{\lambda}(\mu) \in \mathcal B_{2}$, the set of Hilbert-Schmidt operators in $L^{2}_{a}$. We will also define the respective two-modified Fredholm determinants,
\begin{equation}\label{dfnDlm}
\mathcal D(\lambda,\mu)=\det_{2}(I_{L^{2}_{a}}-K_{\lambda}(\mu)), \quad Re(\lambda) \neq 0, \mu \in \mathbb C \backslash \sigma(-\Delta),
\end{equation} 
see, e.g. \cite{GK69, Si05}. For each $\lambda$, this will allow us to characterize the eigenvalues of $A_{\lambda}$ as zeros of the holomorphic function $\mathcal D(\lambda,\cdot)$ of the spectral parameter $\mu$. Using the relation between the spectra of the linearized Euler operator $L_{\rm vor}$ and $A_{\lambda}$ described in Proposition \ref{LvorAlambda}, we will characterize the nonimaginary eigenvalues of $L_{\rm vor}$ as zeros of the function $\mathcal D(\cdot, 0)$ holomorphic in $\lambda $ for $Re(\lambda) \neq 0$. These two results, given in Proposition \ref{alambdaklmueq} and Theorem \ref{newequivalence} constitute the main results in this paper.

In addition, we will introduce a family of Birman-Schwinger type operators $K_{0}(\mu)$, $\mu \in \mathbb C \backslash \sigma(-\Delta)$ associated with Lin's operator $A_{0}$, defined in Definition \ref{dfnA0}, and $-\Delta$, and extend the definition of $\mathcal D(0,\mu)$ to $\lambda=0$ by the same formula \eqref{dfnDlm}. This will allow us to characterize the eigenvalues of $A_{0}$ as zeros of the holomorphic function $\mathcal D(0,\cdot)$ of the spectral parameter $\mu$. It turns out that $K_{\lambda}(\mu) \to K_{0}(\mu)$ in $\mathcal B_{2}$ as $\lambda \to 0^{+}$ uniformly in $\mu$ on compact subsets of $\mathbb C \backslash \sigma(-\Delta)$, see Lemma \ref{lemmaconv} below. Using these facts, we will show that the number of negative eigenvalues of $A_{0}$ is equal to the number of negative eigenvalues of $A_{\lambda}$ provided $\lambda > 0$ is small enough, cf. Proposition \ref{A0Alambd}. As we have already explained in Remark \ref{STP}, this implies the conclusion of Lin's Theorem \ref{LinThm}, thus providing a new proof of this important assertion. 

Our first task is to define the Birman-Schwinger operators $K_{\lambda}(\mu)$. We recall \eqref{defAl2} for $A_{\lambda}$,
\begin{equation}\label{new4defAla}
A_{\lambda}=-\Delta-g'(\psi^0)+g'(\psi^0)\lambda(\lambda-L^0)^{-1},
\, \Re(\lambda)\neq0.
\end{equation}
Notice that if $\mu \in \mathbb C\backslash \sigma(-\Delta)$, then 
\begin{align}\label{alklequiv}
A_{\lambda}-\mu&=(-\Delta-\mu)-g'(\psi^0)+g'(\psi^0)\lambda(\lambda-L^0)^{-1}\nonumber \\&=\bigg(I-\big(g'(\psi^0)-g'(\psi^0)\lambda(\lambda-L^0)^{-1}\big)(-\Delta-\mu)^{-1}\bigg)(-\Delta-\mu) \nonumber\\&= (I-K_{\lambda}(\mu))(-\Delta-\mu), 
\end{align}
where we introduce the operators $K_{\lambda}(\mu)$ by the formula
\begin{equation}\label{klmudef}
K_{\lambda}(\mu)=\big(g'(\psi^0)-g'(\psi^0)\lambda(\lambda-L^0)^{-1}\big)(-\Delta-\mu)^{-1}, \Re(\lambda) \neq 0, \mu \in \mathbb C \backslash \sigma(-\Delta).
\end{equation}
Replacing $\lambda(\lambda-L^0)^{-1}$ in \eqref{klmudef} by $(\lambda-L^{0}+L^{0})(\lambda-L^0)^{-1}$, we see that an alternate expression for \eqref{klmudef} is given by the formula
\begin{equation}\label{klmudef2}
K_{\lambda}(\mu)=-g'(\psi^{0})L^{0}(\lambda-L^{0})^{-1}(-\Delta-\mu)^{-1}, \Re(\lambda) \neq 0, \mu \in \mathbb C \backslash \sigma(-\Delta).
\end{equation}
We notice that $K_\lambda(\mu)\in\cB_2$, the class of Hilbert-Schmidt operators. Indeed, the operators
$g'(\psi^0)$ and $L^0(\lambda-L^0)^{-1}$ are bounded, and the operator $(-\Delta-\mu)^{-1}$ is in $\cB_2$ as it is similar via the Fourier transform to the diagonal operator $\diag\{(|\bk|^{2}-\mu)^{-1}\}_{\bk\in\bbZ^2\setminus\{0\}}$, and the series $\sum_{\bk\in\bbZ^2\setminus\{0\}}(|\bk|^{2}-\mu)^{-2}$ converges. Since $K_\lambda(\mu) \in\cB_2(L^2)$, the following  $2$-modified determinant exists: 
\begin{equation}\label{new455}
\mathcal D(\lambda, \mu)=\det_{2}(I_{L^{2}_{a}}-K_{\lambda}(\mu))=\prod_{n=1}^\infty\big((1-\varkappa_\lambda^{(n)})e^{\varkappa_\lambda^{(n)}}\big),
\end{equation} 
where $\varkappa_\lambda^{(n)}$ denote the eigenvalues of the operator $K_\lambda(\mu)$, and we remark that $\varkappa_\lambda^{(n)}\to0$ as $n\to\infty$. We refer to \cite[Chapter 9]{Si05} or \cite[Sec. IV.2]{GK69} or \cite[Sec. I.7]{Y92} for properties of the two-modified determinants. 

We note that for each fixed $\lambda \in \mathbb C \backslash \sigma(-\Delta)$ the function $K_{\lambda}:\mathbb C \backslash \sigma(-\Delta) \to \mathcal B_{2}$ of the parameter $\mu$ is holomorphic with the derivative 
\begin{equation}
\frac{d K_{\lambda}}{d\mu}=K_{\lambda}(\mu)(-\Delta-\mu)^{-1},
\end{equation}
and for each fixed $\mu \in \mathbb C \backslash \sigma(-\Delta)$ the function $K_{(\cdot)}(\mu):\mathbb C \backslash i \mathbb R \to \mathcal B_{2}$ is a holomorphic function of the parameter $\lambda$. We also note that $0 \notin \sigma(-\Delta)$ because the Laplace operator is being considered on the space $L^{2}_{a}$ of functions with zero average, and is therefore similar via Fourier transform to the diagonal operator $\diag\{\|\mathbf k\|^{2}\}_{\mathbf k \in \mathbb Z^{2}\backslash \{0\}}$. Thus the operators 
\begin{equation}\label{defklm0}
K_{\lambda}(0)=(g'(\psi^{0})-g'(\psi^{0})\lambda(\lambda-L^{0})^{-1})(-\Delta)^{-1}, \quad \Re(\lambda)\neq0,
\end{equation}
are well defined. 
We have the following version of the Birman-Schwinger principle.
\begin{proposition}\label{alambdaklmueq}
Assume Hypothesis \ref{4hyp} and recall formulas \eqref{new4defAla}, \eqref{klmudef} and \eqref{new455}. The following assertions are equivalent for each $\lambda \in \mathbb C \backslash i \mathbb R$ and $\mu \in \mathbb C \backslash \sigma(-\Delta)$:
\begin{align}
&(i)\quad \mu \in \sigma(A_{\lambda})\backslash \sigma(-\Delta),\\ 
&(ii)\quad  1 \in \sigma(K_{\lambda}(\mu)),\\
&(iii)\quad \mathcal D(\lambda,\mu):= \det_{2}(I_{L^{2}_{a}}-K_{\lambda}(\mu))=0.
\end{align}
\end{proposition}
\begin{proof}
The equivalence of (i) and (ii) is a direct consequence of the formula \[A_{\lambda}-\mu=(I-K_{\lambda}(\mu))(-\Delta-\mu) \mbox{ for } \mu \in \mathbb C \backslash \sigma(-\Delta),\] see formula \eqref{alklequiv}. The equivalence of (ii) and (iii) is just a general fact from the theory of two-modified determinants, see, e.g. \cite[Theorem 9.2(e)]{Si05}.
\end{proof}
We are ready to formulate the main result of this section, Theorem \ref{newequivalence}, which gives a characterization of the unstable eigenvalue of $L_{\rm vor}$ in terms of the zeros of the function $\mathcal D(\lambda,0):=\det_{2,L^{2}}(I-K_{\lambda}(0))$ and thus gives us a new way of detecting unstable eigenvalues of the operator $L_{\rm vor}$. In particular, it reduces the study of instability of the linearized vorticity operator $L_{\rm vor}$ to the study of the 2-modified determinant $\det_{2, L^2}(I-K_\lambda(0))$ associated with the operator $K_{\lambda}(0)$ from \eqref{defklm0}.

\begin{theorem}\label{newequivalence}
The following assertions are equivalent:
\begin{align}
&(i) \quad \lambda \in \sigma(L_{\rm vor}) \backslash i\mathbb R,\\
&(ii) \quad 1 \in \sigma(K_{\lambda}(0)),\\ 
&(iii) \quad \mathcal D(\lambda,0):=\det_{2, L^2}(I-K_\lambda(0))=0.
\end{align}
\end{theorem}
\begin{proof}
We will offer two proofs of the main conclusion of the theorem. We note that the equivalence of (ii) and (iii) is just a general fact from the theory of two-modified determinants, see again \cite[Theorem 9.2(e)]{Si05}. Our first proof shows that (i) is equivalent to (iii) while the second proof shows that (i) is equivalent to (ii). 

{\em First proof} (i) $\iff$ (iii). By Proposition \ref{LvorAlambda} we have that $\lambda \in \sigma(L_{\rm vor})$ if and only if $0 \in \sigma(A_{\lambda})$. We now apply the equivalence of assertions (i) and (iii) of Proposition \ref{alambdaklmueq} for $\mu=0$ and conclude that assertions (i) and (iii) of the current theorem are equivalent. 

{\em Second proof} (i) $\iff$ (ii).  We return to representation \eqref{lvorloreln} of the operator $L_{\rm vor}$ and discuss a Birman-Schwinger operator associated with the unperturbed operator $L^0$ and the perturbed operator $L_{\rm vor}$. Recalling that $\omega$ and $\psi$ are related via $\omega=-\Delta\psi$, equation \eqref{lvorloreln} yields
\begin{equation}\label{vortform}
L_{\rm vor}\omega=L^0\big(I+g'(\psi^0(x,y))\Delta^{-1}\big)\omega,
\end{equation}
thus implying, for any $\lambda\in\bbC\setminus\sigma(L^0)$, that 
\begin{equation}\label{comm}\begin{split}
\lambda-L_{\rm vor}&=\lambda-L^0-g'(\psi^0(x,y))L^0\Delta^{-1}\\
&=\big(I-g'(\psi^0(x,y))L^0\Delta^{-1}(\lambda-L^0)^{-1}\big)(\lambda-L^0)\\
&=(I-\widetilde{K_{\lambda}}(0))(\lambda-L^0),
\end{split}
\end{equation}
where we temporarily denote $\widetilde{K_{\lambda}}(0)=g'(\psi^0(x,y))L^0\Delta^{-1}(\lambda-L^0)^{-1}$. We write $A=g'(\psi^0(x,y))L^0\Delta^{-1}$ and $B=(\lambda-L^0)^{-1}$ and note that $\widetilde{K_{\lambda}}(0)=AB$ while, using \eqref{klmudef2} for $\mu=0$, we also have $K_\lambda(0)=BA$ since the operators $g'(\psi^0(x,y))L^0$ and $(\lambda-L^0)^{-1}$ commute.
By standard results, we have that the operator $I-\widetilde{K_{\lambda}}(0)$ is invertible if and only if the operator $I-K_{\lambda}(0)$ is invertible and $\sigma(K_\lambda(0))\backslash\{0\}=\sigma(\widetilde{K_{\lambda}}(0))\backslash\{0\}$.  

Since
$\lambda-L_{\rm vor}=(I-\widetilde{K_{\lambda}}(0))(\lambda-L^0)$,
 for $\lambda \in \mathbb C \backslash \sigma(L^{0})$ and $\sigma(L^{0}) \subset i \mathbb R$, we know that  $\lambda \in \sigma(L_{\rm vor}) \backslash i\mathbb R$ if and only if $1 \in \sigma(\widetilde{K}_{\lambda}(0))$. This
shows that assertions (i) and (ii) of the current theorem are equivalent.
\end{proof}
We will conclude this section by completing the new proof of Lin's Theorem \ref{LinThm} outlined in Remark \ref{STP}. As indicated in this remark, the only missing part of the proof is the following assertion. 
\begin{proposition}\label{A0Alambd}
Assume that the operator $A_{0}$ has an odd number of negative eigenvalues and no kernel. Then $A_{\lambda}$ has the same number of negative eigenvalues provided $\lambda >0$ is small enough. 
\end{proposition}
In order to prove this proposition (and thus finish the proof of Lin's Theorem \ref{LinThm}) we will need to involve Birman-Schwinger type operators $K_{0}(\mu) \in \mathcal B_{2}$ associated with the operator $A_{0}$ and $-\Delta$, and the respective determinant $\mathcal D(0,\mu)=\det_{2,L^{2}}(I-K_{0}(\mu))$. Recalling formula \eqref{defA0}, 
\begin{equation}\label{newdfA0ag}
A_0=-\Delta-g'(\psi^0(x,y))+g'(\psi^0(x,y))P_0,
\end{equation}
similarly to \eqref{alklequiv} we arrive at the formula
\begin{equation}
A_{0}-\mu=(I-K_{0}(\mu))(-\Delta-\mu),
\end{equation}
where we introduce the operator $K_{0}(\mu)$ as follows:
\begin{equation}\label{K0mu}
K_{0}(\mu)=(g'(\psi^{0})-g'(\psi^{0})P_{0})(-\Delta-\mu)^{-1}, \quad \mu \in \mathbb C \backslash \sigma(-\Delta).
\end{equation}
We notice that the only difference with \eqref{klmudef} is that the operator $\lambda(\lambda-L^{0})^{-1}$ is replaced by $P_{0}$. As before, it is easy to see that $K_{0}:\mathbb C \backslash \sigma(-\Delta) \to \mathcal B_{2}$ is a holomorphic function and $d K_{0}/ d \mu = K_{0}(\mu)(-\Delta-\mu)^{-1}$. Similarly to Proposition \ref{alambdaklmueq} one shows the following fact. 
\begin{proposition}\label{K0muA0}
Assume Hypothesis \ref{periodicorbitsassump} and recall formulas \eqref{newdfA0ag} and \eqref{K0mu}. The following assertions are equivalent for $\mu \in \mathbb C \backslash \sigma(-\Delta)$:
\begin{align}
&(i) \quad \mu \in \sigma(A_{0})\backslash \sigma(-\Delta), \\
&(ii) \quad 1 \in \sigma(K_{0}(\mu)), \\
&(iii) \quad \mathcal D(0,\mu):= \det_{2,L^{2}}(I-K_{0}(\mu))=0.
\end{align}
\end{proposition}
\begin{remark}\label{mu0L0}
Because $K_{\lambda}(\cdot):\mathbb C \backslash \sigma(-\Delta) \to \mathcal B_{2}$ and $K_{0}(\cdot):\mathbb C \backslash \sigma(-\Delta) \to \mathcal B_{2}$ are holomorphic, by the general theory of two modified Fredholm determinants we conclude that the functions $\mathcal D(0,\cdot)$ and $\mathcal D(\lambda,\cdot)$ are holomorphic functions of $\mu \in \mathbb C \backslash \sigma(-\Delta)$. This is seen by applying Assertion IV.1.8 in \cite{GK69} and Lemma 9.1 in \cite{Si05}. 

Recall that $\lambda(\lambda-L^{0})^{-1} \to P_{0}$ strongly in $L^{2}_{a}$ as $\lambda \to 0^{+}$ by Lemma \ref{lem_proj}. Re-writing $K_{\lambda}(\mu)$ and $K_{0}(\mu)$ as $K_{\lambda}(\mu)=T_{\lambda}S(\mu)$ and $K_{0}(\mu)=T_{0}S(\mu)$, where $T_{\lambda}=g'(\psi^{0})-g'(\psi^{0})\lambda(\lambda-L^{0})^{-1}$ for $\lambda >0$ and $T_{0}=g'(\psi^{0})-g'(\psi^{0})P_{0}$, and $S(\mu)=(-\Delta-\mu)^{-1} \in \mathcal B_{2}$, we conclude that $T_{\lambda} \to T_{0}$ strongly in $L^{2}_{a}$ as $\lambda \to 0$. The following fact is a standard result in the theory of Hilbert-Schmidt operators, cf. \cite[Theorem 6.3]{GK69}.
\begin{lemma}\label{lemmaconv}
For the operators $K_{\lambda}(\mu)$ from \eqref{klmudef} and $K_{0}(\mu)$ from \eqref{K0mu} we have
$K_{\lambda}(\mu) \to K_{0}(\mu)$ in $\mathcal B_{2}$ as $\lambda \to 0^{+}$ uniformly for $\mu $ from any compact set in $\mathbb C \backslash \sigma(-\Delta)$. 
\end{lemma} 
\begin{proof}
For each fixed $\mu$ assertions $T_{\lambda} \to T_{0}$ strongly as $\lambda \to 0^{+}$ and $S(\mu) \in \mathcal B_{2}$ yield $T_{\lambda}S(\mu) \to T_{0}S(\mu)$ in $\mathcal B_{2}$ by e.g., Theorem 6.3 in \cite{GK69}. The proof given in \cite{GK69} is by writing $S(\mu)=K+L$ where $K$ is of finite rank and $L \in \mathcal B_{2}$ with a small $\mathcal B_{2}$ norm, and then using that $T_{\lambda}u_{j} \to T_{0}u_{j}$ as $\lambda \to 0^{+}$ for a finite basis $\{u_{j}\}$ of the range of $K$. This proof can be easily adapted for $S=S(\mu)$ depending on $\mu$ by taking first a finite $\varepsilon$-dense net $\{\mu_{k}\}$ in the compact set so that $\|S(\mu)-S(\mu_{k})\|_{\mathcal B_{2}}$ are sufficiently small, and next applying the proof of \cite[Theorem III.6.3]{GK69} for each $\mu_{k}$.
\end{proof}
By \cite[Theorem IV.2.1]{GK69} the map $K \to \det_{2}(I-K)$ is continuous on $\mathcal B_{2}$. By Lemma \ref{lemmaconv} we then conclude that $\mathcal D(\lambda, \mu) \to \mathcal D(0,\mu)$ as $\lambda \to 0^{+}$ uniformly for $\mu$ on compact subsets of $\mathbb C \backslash \sigma(-\Delta)$; here $\mathcal D(0,\mu)$ is defined in (iii) of Proposition \ref{K0muA0}.
In what follows, we will apply the argument principle to the family of holomorphic functions $\mathcal D({\lambda},\cdot)$, $\lambda \geq 0$, in parameter $\mu$. The following general formula for the logarithmic derivative can be found in \cite[sec 1.7]{Y92}, see there formula (18),
\begin{equation}\label{Yaffor}
\frac{d}{d\mu}\log \mathcal D(\lambda,\mu)=-\Tr\bigg((I-K_{\lambda}(\mu))^{-1}K_{\lambda}(\mu)\frac{dK_{\lambda}(\mu)}{d\mu}\bigg).
\end{equation}
We recall that $K_{\lambda}(\mu)$ and $\frac{dK_{\lambda}(\mu)}{d\mu}$ are in $\mathcal B_{2}$ and their product is in $\mathcal B_{1}$, see e.g. \cite[Prop I.6.3]{Y92}.
This formula holds for $\lambda \geq 0$ and $\mu \in \mathbb C \backslash \sigma(-\Delta)$. Since $\Tr$ is a continuous functional on $\mathcal B_{1}$, it follows from Lemma \ref{lemmaconv} that
\begin{equation}\label{converD}
\frac{d}{d\mu} \log \mathcal D(\lambda,\mu) \to \frac{d}{d\mu}\log \mathcal D(0,\mu) \text{ as } \lambda \to 0^{+}
\end{equation}
on compact subsets of $\mathbb C \backslash \sigma(-\Delta)$. We recall that by the argument principle, the number of zeros of a holomorphic function enclosed by the contour is computed via the integral over the contour of its logarithmic derivative.
\end{remark}
\begin{lemma}\label{newlem}
Let $\mathcal R$ be a bounded closed rectangle in the $\mu$-plane whose boundary belongs to the resolvent set of the operator $A_{0}$. There exists $\lambda_{0}>0$ such that for all $\lambda \in (0,\lambda_{0})$ the number of eigenvalues of $A_{0}$ in $\mathcal R$ is equal to the number of eigenvalues of $A_{\lambda}$ in $\mathcal R$.
\end{lemma}
\begin{proof}
By Proposition \ref{alambdaklmueq} (iii) and Proposition \ref{K0muA0} (iii) for all $\lambda \geq 0$ the eigenvalues of $A_{\lambda}$ are exactly the zeros of the function $\mathcal D(\lambda,\cdot)$. If $\lambda_{0}$ is small enough and $\lambda \in (0,\lambda_{0})$ then $\mathcal D(\lambda,\mu) \neq 0$ for $\mu \in \partial \mathcal R$, the boundary of $\mathcal R$, because $\mathcal D(0,\mu) \neq 0$ for $\mu \in \partial \mathcal R$ as $\partial \mathcal R$ does not intersect $\sigma(A_{0})$. Thus the logarithmic derivative of $\mathcal D(\lambda,\mu)$ is well defined for $\mu \in \partial \mathcal R$. Using \eqref{converD}, we can make the logarithmic derivative of $\mathcal D(\lambda,\mu)$ arbitrarily close to that of $\mathcal D(0,\mu)$ for $\mu \in \partial \mathcal R$ by further reducing $\lambda_{0}$. By the argument principle the number of zeros of $\mathcal D(0,\mu)$ in $\mathcal R$ is equal to the number of zeros of $\mathcal D(\lambda,\mu)$ in $\mathcal R$ provided $\lambda \in (0,\lambda_{0})$ and $\lambda_{0}$ is small enough.
\end{proof}
We are ready to prove Proposition \ref{A0Alambd}.
\begin{proof}
Since $0$ is not an eigenvalue of $A_{0}$ by assumption, there is a small $\varepsilon >0$ such that the rectangle $\mathcal R_{0}:=[-\varepsilon,\varepsilon]\times[-\varepsilon,\varepsilon]$ does not contain eigenvalues of $A_{0}$. Reducing $\varepsilon$ further, if needed, find a long thin rectangle $\mathcal R=[-a,\varepsilon]\times[-\varepsilon,\varepsilon], a>0$ that contains all the negative eigenvalues of $A_{0}$ and does not contain any other points in $\sigma(A_{0})$. By Lemma \ref{newlem}, if $\lambda_{0}>0$ is small enough and $\lambda \in (0,\lambda_{0})$ then $A_{\lambda}$ has no spectrum in $\mathcal R_{0}$ and the number of eigenvalues of $A_{\lambda}$ in $\mathcal R$ is equal to the number of eigenvalues of $A_{0}$ in $\mathcal R$. Since the number of negative eigenvalues of $A_{0}$ is odd by assumption, we conclude that $A_{\lambda}$ has the same odd number of negative eigenvalues provided $\lambda \in (0,\lambda_{0})$ and $\lambda_{0}>0$ is small enough. This finishes the proof of Proposition \ref{A0Alambd} and completes the proof of Lin's Theorem \ref{LinThm}.
\end{proof}

\section{An application of Lin's instability theorem}
In this section, we study an example of a steady state for which the operator $A_{0}$ is invertible and has an odd number of negative eigenvalues, and is thus unstable by Lin's Theorem \ref{LinThm}. A similar example was used by Lin, see \cite[Example 2.7]{ZL03}, to study instability of steady states on the channel, and we adapted this to the case of the torus. 
\begin{example}
We work with steady states of the form, $\bu^{0}=(\cos my,0)$, where $m \geq 2$ is an integer. Observe that $\psi^{0}(x,y)=\sin my / m$ and $\omega^{0}(x,y)=m \sin my$, from which we obtain that $\omega^{0}(x,y)=g(\psi^{0}(x,y))= m^{2}\psi^{0}(x,y)$, i.e $g'=m^{2}$. 
Let us fix a positive integer $k > 0$ to be determined later. Let $j$ be an integer in $[1,\lfloor m /2\rfloor]$, where $\lfloor m/2 \rfloor$ denotes the greatest integer less than or equal to $m/2$. Denote by $X_{k,j}$ the subspace of  $L^{2}_{a}(\bbT^{2};\bbC)$ consisting of functions of the form $\phi(y)e^{ikx}$ where $\phi(y) \in  Y := \mbox{ span }\{\sin(ny)\}$, %
where $n=j+mp$, $p \in \mathbb Z$. In other words, we have that $X_{k,j}=Y \otimes E_{ikx}$, where we have denoted the one dimensional subspace by $E_{ikx}:= \mbox{ span }\{e^{ikx}\}$. Noting that $\cos (my) \sin(ny) \in Y$ by the trigonometric formula $\cos (my) \sin(ny)= \frac{1}{2}(\sin(n+m)y+\sin(n-m)y)$, it is easy to verify that the space $X_{k,j}$ is invariant under the operators $L^{0}, L_{\rm vor}, A_{\lambda}$, which shall henceforth be restricted to the subspace $X_{k,j}$.

We now compute the kernel of the operator $L^{0}$. Let $\psi(x,y) \in \mbox{ dom } L^{0} \cap X_{k,j}$, i.e., $\psi$ is of the form $\psi(x,y)=\phi(y)e^{ikx}$, where $\phi(y) \in Y \cap \mbox{ dom } L^{0}$. Then
\[L^{0}\psi(x,y) = -\bu^{0}\cdot \nabla \psi(x,y) = -\cos my \partial_{x} \psi(x,y)= -\cos my (ik) \phi(y) e^{ikx}.\] Note that $-\cos my (ik) \phi(y) e^{ikx}=0$ if and only if $\phi(y)=0$. In particular, we see that the kernel of the operator $L^{0}$ acting on $X_{k,j}$ is simply $\{0\}$. 
Thus the orthogonal projection $P_{0}$ onto the kernel of $L^{0}$ restricted to the subspace $X_{k,j}$, is the zero operator, i.e $P_{0}|_{X_{k,j}} = 0$. %

We now compute the operator $A_{0}$, see \eqref{defA0},
\begin{align*}
&A_0 \psi (x,y)=[-\Delta-g'(\psi^0(x,y))+g'(\psi^0(x,y))P_0]\psi(x,y) \\
&= [-\phi''(y) +k^{2}\phi(y) -m^{2}\phi(y)]e^{ikx}=e^{ikx}[A_{0,y}\phi(y)]
\end{align*}
(since $P_{0}=0$), where $A_{0,y}$ acts on the space $Y$, and is given by \[A_{0,y}\phi(y)=-\phi''(y) +k^{2}\phi(y) -m^{2}\phi(y).\] We can write the operator $A_{0}$ as $A_{0,y}\otimes I$ where $A_{0,y}$ acts on the space $Y$ times the identity operator acting on $e^{ikx}$.
The spectrum of $A_{0}$ acting on the space $X_{k,j}$ thus corresponds to the spectrum of the operator $A_{0,y}$ on the space $Y$, that is, $\sigma(A_{0};X_{k,j})=\sigma(A_{0,y};Y)$. If we can prove that $A_{0}$ acting on the space $X_{k,j}$ has an odd number of negative eigenvalues and no kernel, $L_{\rm vor}$ has an unstable eigenvalue by Theorem \ref{LinThm}. Since $\sigma(A_{0};X_{k,j})=\sigma(A_{0,y};Y)$, this amounts to showing that $A_{0,y}$ acting on the space $Y$ has an odd number of negative eigenvalues and no kernel.

We now choose $k$ such that the operator $A_{0,y}$ acting on $Y$ has exactly one negative eigenvalue and no zero eigenvalue. Since $-(\sin(ny))''=n^{2}\sin(ny)$, we have that the eigenvalues contributed by the second derivative operator acting on $Y$ are $\{n^{2}:n=j+mp, p \in \bbZ\}$ (note that $\{\sin(ny)\}$ is a basis for $Y$). Our choice is such that the two smallest eigenvalues in magnitude, contributed by the second derivative operator are when $n=j$ (i.e., when $p=0$) and $n=j-m$ (when $p=-1$). Since $A_{0,y}$ is the second derivative operator added to a multiplication operator, the spectrum of $A_{0,y}$ is given by $\{n^{2}+k^{2}-m^{2}:n=j+mp, p \in \bbZ\}$.

If we choose $k$ such that $m^2-j^2 > k^2 >m^2 -(j-m)^2 $ holds, then the operator $A_{0,y}$ acting on $Y$ has \textit{exactly} one negative eigenvalue (counting multiplicity) and no zero eigenvalue. For a concrete example, choose $m=4$, $j=1$ and $k=3$. Another example is given by the choices, $m=7$, $j=2$ and $k=6$. This gives exactly one negative eigenvalue and no zero eigenvalue for $A_{0,y}$, hence for $A_{0}$, and therefore proves the existence of a positive eigenvalue of $L_{\rm vor}$.

\end{example}



\end{document}